\documentclass[preprint,12pt]{elsarticle}

\usepackage{amscd,amssymb,graphics,lineno}
\usepackage{mathrsfs} 
\usepackage{amsmath}

\newtheorem{theorem}{Theorem}[section]
\newtheorem{corollary}[theorem]{Corollary} 
\newtheorem{lemma}[theorem]{Lemma}
\newtheorem{proposition}[theorem]{Proposition}

\newdefinition{definition}[theorem]{Definition}
\newdefinition{remark}[theorem]{Remark}
\newdefinition{example}[theorem]{Example}
\newproof{proof}{Proof}  

\def\H{{\mathcal H}}

\newcommand{\abs}[1]{\lvert#1\rvert}
\def\norm#1{\left\Vert#1\right\Vert}
\def\norm#1{\left\Vert#1\right\Vert}

\def\C {{\mathbb C}}

\def\H{{\mathcal H}}
\def\R{{\mathbb R}}

\def\e{\epsilon}
\def\tr{{\mathrm{tr}\,}}

\newcommand{\spn}{\rm span}

\newcommand{\M}{\mathfrak{M}}

\newcommand{\SC}{\mathcal{SU}(C)}

\journal{Linear Algebra and Its Applications}

\begin{document}

\begin{frontmatter}

\title{Isometries of a Generalized Numerical Radius on Compact Operators}

\author[label1]{Maria Inez Cardoso Gon\c{c}alves\fntext[label1]{Corresponding author.}\fnref{fn1}}
\ead{minezcg@gmail.com}

\author[label2]{Vladimir G. Pestov\fntext[label2]{Special Visiting Researcher of the program Science Without Borders of CAPES (Brazil), processo 085/2012.}\fnref{fn1,fn2}}
\ead{vpest283@uottawa.ca}

\address[fn1]{Departamento de Matem\'atica - Universidade Federal de Santa Catarina - Trindade - Florian\'opolis - SC - 88.040-900 - Brazil}

\address[fn2]{Department of Mathematics and Statistics, University of
Ottawa, Ottawa, ON, K1N 6N5, Canada}

\begin{abstract}
  We describe all isometries of the $q$-numerical radius on the space ${\mathcal K}(\H)$ of compact operators on a (possibly infinite-dimensional) Hilbert space $\H$.
\end{abstract}


\end{frontmatter}

\section{Introduction}

Let $0<q\leq 1$. The $q$-numerical radius of a bounded linear operator $A$ on a Hilbert space $\H$ is given by
\[r_q(A) = \sup\{\langle Ax,y\rangle\colon \norm x =\norm y = 1, \langle x,y\rangle =q\}.\]
The $q$-numerical radius is a norm on ${\mathcal B}(\H)$, equivalent to the uniform (spectral) norm. 
For $q=1$, this reduces to the classical numerical radius:
\[r(A) = \sup\{\langle Ax,x\rangle\colon \norm x = 1\}.\]

In this article, we describe the isometries of the space ${\mathcal K}(\H)$ of compact operators on an infinite-dimensional complex Hilbert space with regard to the $q$-numerical radius. In the finite-dimensional case, the description was obtained in \cite{maria}, while the case of the classical numerical radius (also in finite dimensions) was previously treated by Le\v snjak \cite{lesnjak}, Li and \v Semrl \cite{LS} and Li et al. \cite{LMR}. 

Here is the main result.

\begin{theorem}
  Let $\H$ be an infinite-dimensional complex Hilbert space.
  Let $0<q\leq 1$ and let $\phi\colon {\mathcal K}(\H)\to {\mathcal K}(\H)$. Then $\phi$ is an isometry of the $q$-numerical radius if and only if there exist a compact operator $S_0\in {\mathcal K}(\H)$, a unitary operator ${\mathcal U}\in {\mathcal{U}}(\H)$ and a complex number $\mu$ with $\abs{\mu}=1$ such that for all $A\in {\mathcal K}(\H)$
  \[\phi(A) =S_0 + \mu U^\ast A^{\dag} U,\]
  where $A^{\dag}$ denotes either $A$ ot $A^t$ or $A^\ast$ or $\bar A$.
  \label{th:main}
\end{theorem}

Here $A^\ast$ is the usual adjoint operator, given by
\[\langle x,Ay\rangle = \langle A^\ast x, y\rangle,\]
$A^t$ is the traspose of $A$,
\[\langle x,A^ty\rangle = \langle \bar y, A \bar x\rangle,\]
and $\bar A$ is the complex conjugate of $A$,
\[\langle x,\bar A y\rangle =\overline{\langle \bar x, A \bar y \rangle}.\]

While the strategy of the proof is broadly similar to that in \cite{maria}, the infinite dimensional case poses numerous difficulties and requires new techniques. For instance, one has to simultaneously deal with a number of different topologies on various spaces of operators, which all coincide between themselves in the finite-dimensional situation. The main challenge was the usage of extreme points, which was the key component of the proofs in \cite{LMR,maria}. If one defines the operator 
\[C_q = q E_{11}+\sqrt{1-q^2} E_{12}\]
on $\H=\ell^2$, then the unit ball of the dual norm to the generalized numerical radius is the convex hull of the saturated unitary orbit 
\[{\mathcal{SU}}\,(C_q) = \{\lambda U^\ast C_q U\colon \lambda\in\C,~ \abs{\lambda}=1,~U\in {\mathcal U}(\H)\}.
\]
It is easy to see that all points of ${\mathcal{SU}}(C_q)$ are extreme points, indeed exposed points of the dual unit ball. However, we were unable to prove that the set of extreme  (or exposed) points of the dual ball is exactly the set ${\mathcal SU}(C_q)$, like it is in a finite-dimensional situation. Thus, a key component of the proofs in \cite{LMR,maria} was missing, and we had to find an ingenious way around it, by using points of weak continuity with regard to the Hilbert-Schmidt norm topology combined with David Milman's converse to the Krein-Milman theorem.

We believe that this technique can be of interest in connection with the next natural open problem: that of characterizing isometries of the generalized numerical radius on a large space ${\mathcal B}(\H)$ of all bounded operators.

Notice also that in order to simplify notation, in the sequel we will often only give proofs for $\H=\ell^2$, but they remain valid for an arbitrary $\H=\ell^2(\Gamma)$.

\section{Generalized numerical radius}
\subsection{Duality}
Let $\H$ be an infinite-dimensional separable Hilbert space.
As usual, we will identify the space ${\mathcal C}_1(\H)$ (the first Schatten class) of all trace class operators on $\H$, equipped with the trace-class norm $\norm T_1 = \tr(\abs T)$, with a predual of the space $B(\H)$, equipped with the uniform operator norm. 

Let us remind how to define a pairing $\langle \cdot,\cdot \rangle$ between ${\mathcal C}_1(\H)$ and $B(\H)$. 
If $C$ is a trace-class operator and $T$ is a bounded linear operator on $\H$, the value of $\langle C,T \rangle$ is given by
\begin{equation}
  \langle C,T\rangle = \tr(C T).
  \label{eq:pairing}
\end{equation}
One can now prove that elements of $B(\H)$ are exactly the bounded linear functionals on ${\mathcal C}_1(\H)$, and that ${\mathcal C}_1(\H)^\ast\cong B(\H)$.

The space ${\mathcal K}(\H)$ of all compact operators on $\H$, equipped with the uniform operator norm, is, under the above duality, the predual of ${\mathcal C}_1(\H)$. In particular, one has
\[{\mathcal K}(\H)^{\ast\ast}={\mathcal B}(\H).\]

For more on the duality and on the above classes of operators, see e.g. \cite{sakai}, 1.15 and 1.19.

In the finite dimensional situation ($\dim\H<\infty$), this duality allows to conveniently identify all the spaces ${\mathcal K}(\H)$, ${\mathcal C}_1(\H)$, ${\mathcal B}(\H)$, as well as its dual space ${\mathcal B}(\H)^\ast$, between themselves. In the infinite dimensions one has to treat them separately. 

\subsection{$C$-Numerical radius}
Fix an operator $C\in {\mathcal C}_1(\H)$. For any $A \in {\mathcal B}(\H)$ the generalized $C$-numerical range of $A$ is defined by:
\[ W_C(A)=\{\tr(CU^*AU) \colon U\in {\mathcal{U}}(\H) \},\]
and the   generalized $C$-numerical radius $r_C$ of $A$ is defined by
\[r_C(A)=\sup\{|\tr(CU^*AU)| \colon U\in {\mathcal{U}}(\H) \}.\]

\begin{theorem} 
  \label{th:r_c}
  $r_C(\cdot)$ is a norm in  ${\mathcal B}(\H)$ if and only if $C$ is non-scalar and $\tr(C)\neq 0$.
\end{theorem}

\begin{proof}
  From the definition of $r_C(A)$ we have that $r_C(A) \geq 0$ for all $A \in {\mathcal B}(\H)$, so we need to show that if $r_C(A)=0$ then $A=0$. Suppose $C$ is non-scalar and $\tr(C)\neq 0$, then if $r_C(A)=0$ we have that $\tr(CU^*AU)=0$ for all unitary $U$, this implies that $W_C(A)$ is a singleton. By lemma 4.2 in  \cite{kuzmaLiRodman}, we have that $A$ is a scalar operator, say $A=\lambda I$, where $\lambda$ is a scalar and $I$ denotes the identity operator.

Therefore we have that: $0=\tr(CU^*AU)=\lambda\tr(C)$, since $\tr(C)\neq 0$, this implies that $\lambda = 0$ and therefore $A$ is the zero operator.

One can easily verify that $r_C(\alpha A)=|\alpha| r_C(A)$, for all $A \in {\mathcal B}(\H)$ and $r_C(A+B)\leq r_C(A)+ r_C(B)$, for all $A, B \in {\mathcal B}(\H)$.
\end{proof}

Let us put the generalized numerical radius in the context of duality. Clearly,
\[r_C(A) = \sup_T {\mathrm Re}\,\langle C,A\rangle,\]
where the supremum is taken over the saturated unitary orbit of $A$ in ${\mathcal C}_1(\H)$:
\[{\mathcal{SU}}(A)=\{\lambda U^\ast A U\colon \lambda\in\C,~\abs\lambda =1,~U\in {\mathcal{U}}(\H)\}.\]
That is to say,
\[r_C(A) = \sup_{\begin{matrix}\abs\lambda =1 \\ U\in {\mathcal{U}}(\H)\end{matrix}} {\mathrm{Re}}(\tr(\lambda  U C^\ast  U^\ast A)).\]
Equivalently, $r_C(A)$ is the supremum of the linear functional $\langle -,A\rangle$ over the convex circled hull of the unitary orbit of $C$.
In the case where $C$ is non-scalar and $\tr(C)\neq 0$, Theorem \ref{th:r_c} implies that this hull is absorbing and bounded in the vector space ${\mathcal C}_1(\H)$, so is the unit ball of a certain norm, $r_C^\ast$. The dual norm to $r_C^\ast$ is the norm $r_C$ on ${\mathcal B}(\H)$.

We do not know if in this generality the norm $r_C^\ast$ is the dual norm to the restriction of $r_C$ to ${\mathcal K}(\H)$, nor whether the norm $r_C$ is equivalent to the uniform operator (spectral) norm on ${\mathcal B}(\H)$. In particular, we cannot assert that the norms $r_C$ and $r_C^\ast$ are complete. However, the answers to all of these questions are positive in the special case of a $q$-numerical radius.

\subsection{$q$-Numerical radius}

In the important particular case where $\H=\ell^2$ and $C=E_{11}$ one recovers the classical numerical radius:
\[r(A) = \sup_{\norm x =1}\abs{\langle Ax, x\rangle}.\]
It is well known that the spectral radius is a norm on ${\mathcal B}(\H)$ equivalent to the uniform norm, moreover
\[r(A)\leq \norm A \leq 2 r(A).\]
(See e.g. Th. 2.14 in \cite{kubrusly}.) 

The case of interest for us is a more general case of the $q$-numerical radius, where $0<q\leq 1$:
\[r_q(A) = \sup_{\norm x =\norm y =1,\langle x,y\rangle =q}\abs{\langle Ax, y\rangle}.\]
In this case, the matrix $C$ assumes the form
\[C_q=qE_{11}+\sqrt{1-q^2}E_{12}.\]

Geometric considerations in a two-dimensional Euclidean space imply the following.

\begin{lemma}
  For all $A\in {\mathcal B}(\H)$,
  \[  r_q(A)\geq q r(A).\]
\end{lemma}

\begin{proof}
  Let $\norm\xi=1$ and $\langle A\xi,\xi\rangle =\pm\alpha$, where $\alpha>0$. Identify $\xi$ with the second coordinate vector $e_2$ of the plane $\R^2$ spanned by $\xi,A\xi$, so that $A\xi$ belongs to either the first or the fourth quadrant ($\langle A\xi,e_1\rangle \geq 0$). Thus, $A\xi = (\sqrt{1-\alpha^2},\pm\alpha)$.
  Choose $\zeta$ with $\norm\zeta=1$, $\langle \xi,\zeta\rangle = q$, and $\zeta$ belonging to the first quadrant if so does $A\xi$, or the second quadrant if $A\xi$ is in the fourth. One can write $\zeta = (\pm\sqrt{1-q^2},q)$. Now one has
  \begin{eqnarray*}
    \lvert\langle A\xi,\zeta\rangle\rvert &=& \lvert \pm\sqrt{1-\alpha^2}\sqrt{1-q^2} \pm \alpha q\rvert \\
    &=& \sqrt{1-\alpha^2}\sqrt{1-q^2} + \alpha q,
    \end{eqnarray*}
  whence the result follows.
\end{proof}

We conclude: $r_q(A)\geq (q/2)\norm A$. At the same time it is evident that $r_q(A)\leq \norm A$. Thus, the norm $r_q$ on ${\mathcal B}(\H)$ is equivalent to the uniform (spectral) norm. 

As a consequence, ${\mathcal K}(\H)$ is weakly dense in ${\mathcal B}(\H)$ relative to the $q$-numerical radius (because the same is true of the uniform operator norm).
It follows that the norm $r_q^\ast$ on ${\mathcal C}_1(\H)$ as defined earlier is the dual norm to the restriction of $r_q$ on ${\mathcal K}(\H)$.
In particular, $r_q^\ast$ is equivalent to the trace class norm on ${\mathcal C}_1(\H)$. We also conclude that the norm $r_q$ on ${\mathcal B}(\H)$ and on ${\mathcal K}(\H)$ and the norm $r_q^\ast$ on ${\mathcal C}_1(\H)$ are  complete. 

On the space of compact operators, sharper equivalence constants follow from the results of \cite{LMR}.

\begin{theorem}
  The norm $r_q$ on ${\mathcal K}(\H)$ is equivalent to the uniform operator norm, with the constants
  \begin{equation}
    \label{eq:eq}
    r_q(A)\leq \norm A \leq \beta r_q(A),~~A\in {\mathcal B}(\H),\end{equation}
  where
  \[\beta=\begin{cases}
    \max\{1/p,1/q\}&\mbox{ if }\frac 12\leq p,\\
    \sqrt{5-4p}/\abs q,&\mbox{ if } \frac 14\leq p <\frac 12,\\
    2/ q,&\mbox{ otherwise.}
  \end{cases}\]
\end{theorem}

\begin{proof}
  The above result was established in the case $\dim\H<\infty$ in \cite{LMR}, Theorem 3.5. If we fix an orthonormal basis for $\H$, the inequalities in (\ref{eq:eq}) will hold for all operators $A$ represented by matrices with finitely many entries. The same is true for the operators belonging to the norm completion of the linear space of such operators, that is, ${\mathcal K}(\H)$. 
\end{proof}
 
 \section{Saturated unitary orbit}
 
 \begin{theorem} An operator $A \in {\mathcal B}(\H)$ belongs to ${\mathcal{SU}}(C_q)$ if and only if $A$ is a rank one operator, $|\textrm{tr}(A)|=|q|$ and $\norm{A}_2=1$ (in the Hilbert-Schmidt, or Frobenius, norm).
  \label{th:cq}
\end{theorem}

\begin{proof} If $A \in {\mathcal{SU}}(C_q)$, then there exist a scalar $\theta$, $|\theta|=1$ and an unitary operator $U \in  {\mathcal{U}}(\H)$ such that $A=\theta U^*C_qU$. Since $U$ is a unitary operator and $C_q$ is a rank one operator of Hilbert-Schmidt norm 1 and $\textrm{tr}(C_q)=q$, we have that $A$ is a rank one operator, $|\textrm{tr}(A)|=|q|$ and $||A||=1$. 

  Suppose now that $A$ is a rank one operator, $\textrm{tr}(A)=q$ and $\norm{A}_2=1$. Let's show that $A\in {\mathcal{SU}}(C_q)$.

Since $A$ is a rank one operator, there exists $x,y \in \mathcal{H}$ such that $A=x\otimes y^*$.
We can suppose without loss of generality that $||x||=1$. Since $\textrm{tr}(A)=q$, from the definition of trace of a rank one operator we have that $\langle x,y\rangle=q$. Since $1=||A||_2=||x||||y||$, we also have that $||y||=1$.

If $|q|=1$, from $\langle x,y\rangle=q$ we have that there exists $\theta \in \C$, $|\theta|=1$ such that $y=\theta x$. Let $U\in {\mathcal{U}}(\H)$ such that $U^*x=e_1$, then $U^*AU=U^*(x\otimes  y^*)U= \bar{\theta} e_1\otimes e_1^*=\bar{\theta} E_{11}=\bar{\theta} C_1$ and therefore $A\in {\mathcal{SU}}(C_q)$.

Suppose now $|q| <1$. In this case $x$ and $y$ are linearly independent. Let $z=y-\left\langle y,x\right\rangle x=y- q x$. Then $z$ and $x$ are orthogonal and moreover $||z||^2=1-|q|^2$. Once again we can find $U\in {\mathcal{U}}(\H)$ such that $U^*x=||x||e_1=e_1$ and $U^*z=||z||e_2=\sqrt{1-|q|^2}e_2$. Therefore $U^*AU=U^*(x\otimes  y^*)U=e_1\otimes(\sqrt{1-|q|^2}e_2+qe_1)^*=C_q$. Therefore $A \in {\mathcal{SU}}(C_q)$. 
\end{proof}

 \begin{lemma}
    \label{l:rk}
    Let $k\in\mathbb N$. The set ${\mathfrak R}_k$ of all operators from $\ell^2$ to itself having rank $\leq k$ is closed in the strong operator topology.
\end{lemma}

\begin{proof}
  For a $T\notin {\mathfrak R}_k$, choose $k+1$ orthonormal vectors in the range of $T$: 
  \[T(x_i)=y_i, ~~~i=1,2,\ldots,k+1,y_i\perp y_j,~~i\neq j.\]
  Define an open neighbourhood $V$ of $T$ in the strong operator topology by the condition
  \[\norm{S(x_i) - y_i}<\frac 1{2n}.\]
  For any $S\in V$ the vectors $z_i=S(x_i)$ are linearly independent. Indeed, let us consider a nontrivial linear combination $\sum_{i=1}^n\lambda_i z_i$, assuming also $\sum\lambda_i^2=1$.  Now,
  \begin{eqnarray*}
    \sum_{i=1}^n\lambda_i z_i &=& \sum_{i=1}^n\lambda_i y_i + \sum_{i=1}^n\lambda_i (S(x_i) - y_i).
    \end{eqnarray*}
  The norm of the first vector on the r.h.s. is $1$, while the norm of the second is bounded by
  \[\sum_{i=1}^n\abs{\lambda_i} \norm{S(x_i) - y_i} \leq\frac 12.\]
\end{proof}
 
\begin{proposition}
  The saturated unitary orbit ${\mathcal{SU}}(C_q)$ is closed in the trace class norm on $\mathcal C_1(\H)$.
  \label{p:traceclosed}
\end{proposition}

\begin{proof}
  According to Theorem \ref{th:cq}, the set ${\mathcal{SU}}(C_q)$ is the intersection of the set ${\mathfrak R}_1$ of all rank one operators (closed in the strong operator topology by Lemma \ref{l:rk}, therefore in the trace class norm topology), the unit sphere with regard to the Hilbert-Schmidt norm (which is closed in the Hilbert-Schmidt norm topology, therefore in the finer trace class norm topology), and the level surface of the function $T\mapsto \abs{\tr(T)}$ which is continuous with regard to the trace class norm.
\end{proof}

\begin{lemma}
The trace class metric and Hilbert-Schmidt metric are Lipschitz equivalent on the set ${\mathfrak R}_k$ of all operators of rank $\leq k$, with constants which only depend on $k$: for all $T,S\in {\mathfrak R}_k$,
  \[\norm{T-S}_2\leq \norm{T-S}_1\leq \sqrt{2k}\norm{T-S}_2.\]
  \label{l:frank}
\end{lemma}

\begin{proof}
  Recall that, given a compact operator $T$, one can write
  \begin{equation}
    \label{eq:repr}
    T(x) = \sum_{i=1}^\infty \lambda_i \langle x,v_i\rangle u_i,\end{equation}
  where $(v_i)$ and $(u_i)$ are suitably chosen orthonormal bases of $\H$. Now the trace of $T$ is given by
  \[\tr(T) =\sum_i \lambda_i,\]
  the trace class norm by
  \[\norm{T}_1 =\sum_i\abs{\lambda_i},\]
  and the Hilbert-Schmidt norm by
  \[\norm{T}_2 = \left(\sum_i\abs{\lambda_i}^2\right)^{1/2},\]
  where the above quantities are independent of the choice of representation (\ref{eq:repr}).
  
  If $T,S\in {\mathfrak R}_k$, then the operator $T-S$ has rank $\leq 2k$, and therefore the number of non-zero coefficients $\lambda_i$ in the sum (\ref{eq:repr}) representing $T-S$ is at most $2k$. The values $\norm{T-S}_1$ and $\norm{T-S}_2$ become the values of the $\ell_1$ norm and $\ell_2$ norm respectively of the same vector 
  % (namely, $(\lambda_1,\ldots,\lambda_{2k})$) 
in the standard vector space of dimension $2k$. But these are Lipschitz equivalent, with constants which only depend on the dimension.
\end{proof}

\begin{corollary}
  The trace class norm topology and the Hilbert-Schmidt norm topology coincide on the saturated unitary orbit ${\mathcal{SU}}(C_q)$. Moreover, the corresponding metrics on the orbit are Lipschitz equivalent. 
\end{corollary}

\begin{corollary}
  The saturated unitary orbit ${\mathcal{SU}}(C_q)$ is complete in the Hilbert-Schmidt metric.
\label{c:closedHS}
\end{corollary}

\begin{proof}
  By Corollary \ref{p:traceclosed}, ${\mathcal{SU}}(C_q)$ is complete with regard to the trace class metric, and since the two metrics are Lipschitz equivalent, the same holds for the Hilbert--Schmidt metric.
\end{proof}

We will be dealing with two weak$^\ast$ topologies on the space ${\mathcal C}_1(\H)$: the one coming from the duality with ${\mathcal K}(\H)$, which serves as the weak$^\ast$-topology for the trace class norm (equivalently, for the norm $r_q^\ast$) and which we denote $w^\ast_1$, and the weak$^\ast$-topology for the completion of the space with regard to the Hilbert-Schmidt norm (that is, the weak topology), which we will denote $w^\ast_2$. The $w^\ast_1$-topology is finer than the $w^\ast_2$-topology, because there are more compact operators than Hilbert-Schmidt operators. 

\begin{lemma}
  \label{l:finer}
  At every point of the unit ball of the normed space $({\mathcal C}_1(\H),r_{q}^\ast)$ belonging to the
  saturated unitary orbit ${\mathcal{SU}}(C_q)$ the $w^\ast_2$-topology is finer than the Hilbert-Schmidt topology (and thus, the two topologies coincide).
\end{lemma}

\begin{proof}
  It is well known and easily seen that every point of the unit sphere of a Hilbert space is a point of weak continuity, that is, the neighbourhood filters in the weak topology and in the norm topology coincide at this point. Moreover, the same is true if the neighbourhoods of the point are considered not just on the sphere, but in the entire Hilbert unit ball.
  Denote $B_2$ the unit ball of ${\mathcal C}_1(\H)$ with regard to the Hilbert-Schmidt norm, and $B_q$ the unit ball with regard to $r_{q}^\ast$. If $x$ is a point on the unit sphere of $B_2$, then for each $\e>0$ there is $\delta$ so that whenever $y\in B_2$ and $\langle x,y\rangle >1-\delta$, one has $\norm{x-y}_2<\e$. Now let $x\in {\mathcal{SU}}(C_q)$. Then $x$ is on the unit sphere of $B_2$, and the functional $\phi(y)=\langle x,y\rangle$ is $w^\ast_2$-continuous. If $y\in B_{q}$ and $\phi(y)>\delta$, where $\delta=\delta(\e)$ as above, then $y\in B_2$ and so $\norm{x-y}_2<\e$. 
\end{proof}

  Recall David Milman's (partial) converse to the Krein--Milman theorem (\cite{milman}, Theorem 1): if $T$ is a subset of a convex compact subset of a topological vector space and the closed convex hull of $T$ is all of $K$, then every extreme point of $K$ belongs to the closure of $T$. As an immediate corollary, we obtain:

\begin{lemma}
 All the extreme points of the unit ball of the norm $r_{q}^\ast$ are contained in the $w_1^\ast$-closure of the saturated unitary orbit of $C_q$.
 \label{l:partilconverse}
\end{lemma}

\begin{proof}
  The unit ball of $r_q^\ast$ is the closed convex hull of the saturated unitary orbit ${\mathcal{SU}}(C_q)$. Equipped with the weak$^\ast$ topology relative to the duality between ${\mathcal K}(\H)$ and ${\mathcal C}_1(\H)$, the unit ball is compact. Now apply the converse to the Krein--Milman theorem with $T= {\mathcal{SU}}(C_q)$.
\end{proof}

\begin{remark}
  Even if the topologies (in our case, weak$^\ast$ topology and the trace class topology) coincide on a set (the saturated unitary orbit), one cannot conclude that the closures of this set in the ambient space with regard to the two topologies are the same. We do not know whether the extreme points of the convex closure of the saturated unitary orbit are all contained in the saturated unitary orbit.
\end{remark}
 
\section{Isometry of the generalized numerical radius} 

\subsection{Linearity}
Let $\phi$ is a $q$-numerical radius isometry of $({\mathcal K}(\H),r_q)$. Then $\phi$ is the sum of a translation by some $S_0\in {\mathcal K}(\H)$ and a $q$-numerical radius isometry preserving zero. From now, we will therefore presume that $\phi(0)=0$. The Mazur--Ulam theorem (cf. e.g. Theorem 1.3.5 in \cite{FJ}) says: if $\phi$ is an isometry from a normed linear space $X$ onto a normed linear space $Y$, and if $\phi(0)=0$, then $\phi$ is real linear.

Denote $\psi$ the dual linear map. We conclude: $\psi$ is a real linear isometry of $({\mathcal C}_1(\H),w_q^\ast)$.
 
\subsection{Representing $\psi$ as $P(\cdot)Q$}

The following result is an infinite-dimensional analogue of Lemma 2.2 in \cite{maria}.

\begin{lemma}\label{rankone}
 Let $C$ and $R$ be two bounded rank one operators on $\ell^2$, where
  $\norm{R} < {\rm min \{ 2q, 2p} \}$, and $p=(1-q^2)^{1/2},$
  and let
    $${\mathfrak S}(R) : = \{ A \in \SC : R=A +B \ \ {\rm  for} \
     {\rm some}  \  \  B \in \SC \},$$
 and let  $\M(R) = \spn\, {\mathfrak S}(R).$
 Then $\M(R)$ is a vector space of infinite dimension over $\R$.
 \label{l:2.2}
\end{lemma}

\proof Let $R=u\otimes v^t$. Let´s show that $\M(R) =  {\mathfrak L}_1 +
{\mathfrak L}_2$, where
 $${\mathfrak L}_1 = \{ u\otimes y^* : y \in \ell^2 \}, \  \
 {\mathfrak L}_2 = \{ x\otimes v^* : x \in \ell^2 \}.$$

First assume that $R=A+B$ for operators $A, B$ in $\SC$. If $A=
a\otimes c^t$ and $B=b\otimes d^t$ for $a,b,c,d \in \ell^2$, then either $a$
and $b$ are linearly dependent (in which case, we may take
$a=b=u$), or $c$ and $d$ are linearly dependent  and again we may
take $c=d=v$. This means that either both $A$ and $B$ are in
${\mathfrak L}_1$ or both are in ${\mathfrak L}_2$. Therefore $\M(R)
\subseteq {\mathfrak L}_1 + {\mathfrak L}_2.$

For the converse,  we may replace $R$ by any  operator from its unitary
orbit, hence we may assume that $R= \xi E_{11} + \eta E_{12}$ and
so $u=e_1$.  By assumption, we have $|\xi| < 2q$ and $|\eta |<
2p$. For every $q^{\prime}$ and $p^{\prime}$ with
 $|\xi| /2 < q^{\prime} <q$ and
$|\eta|/2< p^{\prime} <p$, there exist complex numbers
 $z_j$, $1 \leq j \leq 4$, such that
 $|z_1|=|z_2|= q^{\prime}$, $|z_3|=|z_4|= p^{\prime}$. Let
 $r = (1-(p^{\prime})^2 - (q^{\prime})^2)^{1/2},   t \in \R$,
  $k\ge 3$ and
  $$A_t= z_1E_{11} + z_2 E_{12} + re^{it} E_{1k}, \ \
   B_t= z_3E_{11} + z_4 E_{12} - re^{it} E_{1k}.$$ Then
 $A_t, B_t \in \SC$ and $A_t+B_t=R$.
 With all the various choices of $p^{\prime}$, $q^{\prime}$ and
 $t$, it is easy to see that $\M(R)$ contains $E_{1j}$ and
 $iE_{1j}$ for $j\geq 1$, hence it contains $e_1\otimes y^*$ for
 every $y \in \ell^2$. This proves that  ${\mathfrak L}_1 \subset \M(R)$.
 By symmetry we also get ${\mathfrak L}_2 \subset \M(R)$. This proves
 that $\M(R) =  {\mathfrak L}_1 + {\mathfrak L}_2$. Each of ${\mathfrak L}_1$
 and ${\mathfrak L}_2$ is a space of infinite  dimension over $\R$, therefore
 $\M(R)$ has infinite dimension over $\R$.   \endproof

 The following is an infinite-dimensional version of Lemma 2.3 in \cite{maria}.
 
 \begin{lemma}
   Let $R$ be a bounded rank two operator on the Hilbert space $\ell^2$, and let $\mathfrak R={\mathfrak R}_1$ denote the set of all bounded rank one operators on $\ell^2$. Then the set
   \[{\mathfrak A}(R) = \{A\in {\mathfrak R}\colon R = A+B\mbox{ for some }B\in {\mathfrak R}\}\]
   spans a real vector space of dimension $7$ and so $\dim({\mathfrak M}({R}))\leq 7$, where $\mathfrak M (R)$ is defined as in Lemma \ref{l:2.2}.
   \label{l:2.3}
 \end{lemma}
 
 \begin{proof} Without loss in generality, we assume the space $\ell^2$ complex and all the operators complex-linear. If $A=a\otimes b^\ast$ is a rank one bounded linear operator, given by $x\mapsto a\langle b,x\rangle$, then ${\mathrm{range}}\,(A)={\mathrm{span}}\,(a)$ and ${\mathrm{null}}\,(A) = \{b\}^\perp$. Therefore if $A=a\otimes b^\ast$ and $B=c\otimes d^\ast$ are rank one bounded operators, then
   \[{\mathrm{range}}\,(A+B)\subseteq {\mathrm{span}}\,\{a,c\}\mbox{ and }
   ({\mathrm{null}}\,(A+B))^\perp \subseteq {\mathrm{span}}\,\{b,d\}.\]
   If $A+B$ has rank $2$, then both inclusions must be equalities. 
   
   Let now $R$ be a rank two bounded linear operator. Denote ${\mathfrak R} ={\mathrm{range}}\,(R)$ and ${\mathfrak N}={\mathrm{null}}\,(R)$. If $R$ is a sum of two rank one bounded operators $A=a\otimes b^\ast$ and $B=c\otimes d^\ast$, then from above we must have ${\mathfrak R}={\mathrm{span}}\,\{a,c\}$ and ${\mathfrak N}^{\perp} = {\mathrm{span}}\,\{b,d\}$. So the operators $A$ and $B$ vanish on $\mathfrak N$ and map the two-dimensional space ${\mathfrak N}^{\perp}$ into the two-dimensional space $\mathfrak R$. 
   
   This reduces the problem to the case $n=2$, which is treated in an identical way to that in the proof of Lemma 2.3 in \cite{maria}.
 \end{proof}

 Now we can conclude that $\psi$ preserves rank one operators, just like in \cite{maria}, Lemma 2.4, which is the finite-dimensional analogue of our next result.

\begin{lemma}
  Assume that $\psi$ is a bijective bounded real-linear operator on ${\mathcal K}(\ell^2)$ and that $\psi({\mathcal{SU}}\,(C_q))={\mathcal{SU}}\,(C_q)$. Then $\psi$ preserves rank one bounded linear operators.
\end{lemma}

\begin{proof} Let $R\in {\mathcal K}(\ell^2)$ be of rank one and let $R^\prime =\lambda R$, where $0<\lambda <(\min\{2p,2q\}){\norm R}^{-1}$. The operator $R^\prime$ satisfies the norm condition of Lemma \ref{l:2.3}. Thus the space ${\mathfrak M}\,(R^\prime)$ is infinite-dimensional. By real linearity, we have that $\psi$ maps ${\mathfrak M}\,(R^\prime)$ onto ${\mathfrak M}\,(\psi(R^\prime))$. Therefore ${\mathfrak M}\,(\psi(R^\prime))$ is infinite-dimensional as well. By Lemma \ref{l:2.2}, $\psi(R^\prime)$ is not of rank $2$. Since $\psi(R)$, just like $R$, can be represented as a sum of two operators from ${\mathcal{SU}}\,(C_q)$, its rank does not exceed $2$. It is also nonzero since $\psi$ is a bijection. Therefore $\psi(R^\prime)$ and hence also $\psi(R)$ have rank one.
\end{proof}

Now let us establish an analogue of Corollary 2.5 in \cite{maria}.

\begin{lemma}
  There exist operators $P,Q\in GL(\H)$ such that for all $A\in {\mathcal C}_1(\H)$,
  \[\psi(A)=PA^\dag Q,\]
  where $A^\dag$ denotes either $A$ or $A^t$ or $A^\ast$ or $\bar A$.
\end{lemma}

\begin{proof}
  According to Theorem 3.1 in \cite{OS}, there exist bijective linear operators $P,Q$ from $\H$ to itself so that for all $A\in {\mathcal K}(\H)$,
  \[\psi(A) = P\tilde A Q\mbox{ or }\psi(A)= P\tilde A^t Q,\]
  where $\tilde A$ is obtained from $A$ by applying a field automorphism $c\mapsto \tilde c$ entry-wise. Since the only real-linear automorphisms of $\C$ are the identity and complex conjugation, the result will follow once we establish that $P,Q$ are bounded, with bounded inverse.
  
  We will only give an argument in the case $A^\dag = A$, the rest is similar. In this case, for every $x,y\in\H$, one has
  \[\psi(x\otimes y^\ast)=P(x\otimes y^\ast) Q^\ast.\]
  Let $z\in\H$ be a non-zero vector.
  For every $x\in\H$ the operator $x\otimes z^\ast$ is of rank one. The mapping
  \[i_z\colon \H\ni x\mapsto x\otimes z^\ast \in {\mathcal C}_1(\H)\]
  is a linear isometric embedding with regard to the trace class norm if $\norm z=1$. This can be seen by choosing two orthonormal bases $(v_n)$, $(u_n)$ so that $x=\norm x\cdot v_1$ and $u_1=z$:
  \begin{eqnarray*}
    \norm{x\otimes z^\ast}_1 &=& \norm{\norm x \langle -,v_1\rangle u_1 + 0\cdot ...}_1 \\
    &=& \norm x.
    \end{eqnarray*}
  Therefore, if $z\neq 0$, the mapping $i_z$ is an isomorphic embedding of $\H$, admitting an isomorphic inverse. We conclude: 
  \[P=i_{Qz}^{-1}\circ \psi \circ i_z,\]
  and this composition is bounded by the assumption on $\psi$. A similar argument establishes that $Q$ is bounded, and applying the above to $\psi^{-1}$, one concludes that the inverses of $P,Q$ are bounded as well. 
\end{proof}

\begin{corollary}
  The map $\psi$ is an isomorphism of ${\mathcal C}_1(\H)$ with regard to the Hilbert-Schmidt norm, and therefore uniquely extends to an isomorphism of ${\mathcal C}_2(\H)$.
\end{corollary}

\begin{proof} It is enough to apply the property 
  \[\norm{AB}_2\leq\norm{A}\norm{B}_2.\]
\end{proof}

\subsection{Proving that $P,Q$ are unitary}

The following should be obvious.

\begin{lemma}
  \label{l:T}
  Let $X$ be a set equipped with two topologies, $\tau$ and $\sigma$, and let a map $f\colon X\to X$ be a homeomorphism with regard to the both topologies. Denote $T$ the set of all points $x\in X$ at which the neighbourhood filter $\tau$ contains the neighbourhood filter of $\sigma$. Then $T$ is invariant under $f$:
  \[f(T)=T.\]
\end{lemma}

\begin{lemma}
  The saturated unitary orbit ${\mathcal{SU}}(C)$ is invariant under $\psi$.
\end{lemma}

\begin{proof}
  Let $A\in {\mathcal{SU}}(C)$ be any point.
  Since $\psi(A)$ is an extreme point of $B_C$, by the partial converse to the Krein-Milman theorem (cf. Lemma \ref{l:partilconverse}), $\psi(A)$ belongs to the $w^\ast_1$-closure of ${\mathcal{SU}}(C)$. The application $\psi$, being an invertible trace class bounded dual linear map, is a homeomorphism of the unit ball of $r_C^\ast$ with regard to the $w_1^\ast$-topology, and at the same time a homeomorphism the unit ball of $r_C^\ast$ with regard to the Hilbert-Schmidt topology. 
  By Lemmas \ref{l:finer} and \ref{l:T}, the two topologies on this ball must coincide at the point $\psi(A)$. Since every $w_1^\ast$-neighbourhood of $\psi(A)$ contains a point of ${\mathcal{SU}}(C_q)$, the same is true of every Hilbert-Schmidt neighbourhood of the same point, and so $\psi(A)$
  must belong to the Hilbert-Schmidt closure of ${\mathcal{SU}}(C)$ in the unit ball of $r_C^\ast$. But ${\mathcal{SU}}(C)$ is complete in the Hilbert-Schmidt topology by Corollary \ref{c:closedHS}, hence closed in the Hilbert--Schmidt topology on every ambient space, including the ball in question.
\end{proof}

Now we can conclude that $P,Q$ are unitary just like in \cite{maria}, Proposition 2.8, whose proof remains true for the infinite dimensional case with a suitable adjustements, as follows.

\begin{proposition}
  Let $\H$ be an infinite dimensional Hilbert space and let $\psi(A)= PA^\dag Q$ for every $A\in {\mathcal C}_1(\H)$, where $A^\dag$ is either $A$ or $A^t$ or $A^\ast$ or $\bar A$. If $\psi({\mathcal{SU}}(C)) ={\mathcal{SU}}(C)$, then each of $P$ and $Q$ is a scalar multiple of a unitary operator on $\H$ and $PQ=\lambda I$ for a complex number $\lambda$ of modulus one. Consequently, there exists a unitary operator $U\in{\mathcal U}(\H)$ such that $\psi(A)=\lambda U^\ast A U$ for every $A\in {\mathcal C}_1(\H)$.
\end{proposition}

\begin{proof}
  Since each of the maps $A\mapsto A^t$, $A\mapsto \bar A$ and $A\mapsto A^\ast$ preserves the saturated unitary orbit, it suffices to prove the Proposition for the case $\psi(A) = PAQ$.
  
  Let $x,y\in \H$ be such that $x\perp y$ and $\norm x =\norm y =1$. Let $s\in\R$ and $z=qx + e^{is}\sqrt{1-q^2}y$. Then $x\otimes z^\ast$ is a rank one operator of the Hilbert-Schmidt norm $1$ and trace $q$, hence $x\otimes z^\ast\in {\mathcal {SU}}\,(C_q)$ by Theorem \ref{th:cq}. Thus $P(x\otimes z^\ast) Q=\psi(x\otimes z^\ast)\in {\mathcal SU}\,(C_q)$, which implies that
  \begin{enumerate}
  \item[(a)] $\norm{P(x\otimes z^\ast) Q}=1$, amd
  \item[(b)] $\abs{\tr(P(x\otimes z^\ast) Q)}=q$.
  \end{enumerate}
  
  From condition (a), we have $1=\norm{P(x\otimes z^\ast) Q}=\norm{Px}\cdot\norm{Q^\ast z}$, therefore $\norm{qQ^\ast x +e^{is}\sqrt{1-q^2}Q^\ast y}=K$ for all $s\in\R$, where $K$ is a positive constant. Lemma 2.6 in \cite{maria} says that if $u,v$ are two vectors in a complex Hilbert space with the property $\norm{u+e^{is}v}=1$ for all $s\in\R$, then $u\perp v$. We conclude: $Q^\ast x \perp Q^\ast y$.
  
  We conclude that $Q^\ast$ preserves orthogonality. By Lemma 2.7 in \cite{maria}, a bounded linear operator on a Hilbert space is a scalar multiple of isometry if and only if it preserves orthogonality of vectors. We conclude that $Q^\ast$ is a scalar multiple of an isometry of $\H$, and therefore a scalar multiple of a unitary operator on $\H$. Similarly, $P$ is a scalar multiple of a unitary operator.
  
  The trace condition (b) implies that
  \[\left\vert q\langle QPx, x\rangle + \sqrt{1-q^2} e^{is}\langle QPx, y\rangle\right\vert =q,~~\mbox{ for every }s\in\R.\]
  This implies that $\langle QPx,x\rangle =0$ or $\abs{\langle QPx,x\rangle}=1$. Since this is true for every unit vector $x$, we get that $W(QP)$, the classical numerical range for $QP$, is included in the union of the unit circle and $\{0\}$. The well-known convexity of the classical numerical range (the Toeplitz-Hausdorff theorem \cite{hausdorff,toeplitz}) implies that $W(QP)$ is a singleton of modulus $1$ or $0$. But $QP$ is evidently nonzero, hence $QP=\lambda I$, for a complex number of modulus one.
\end{proof}

Now we are ready to establish the main result.
\vskip .2cm

\noindent
{\sc Proof of Theorem \ref{th:main}.}
For the ``if'' part, it is easy to see that
\[r_q(A)=r_q(\bar A)=r_1(A^t)=r_q(A^\ast).\]
  
Conversely, if $\phi$ is an isometry of the normed space $({\mathcal{C}}_1(\H),r_q)$ satisfying $\phi(0)=0$, then $\phi$ is linear, and  there exist a unitary operator $U$ on $\H$ and a complex number $\lambda$ of modulus $1$ such that the dual operator $\psi$ on ${\mathcal C}_1(\H)$ satisfies
\[\psi(A) = \lambda U^\ast A^\dag U\mbox{ for all }A\in {\mathcal C}_1(\H).\]
This implies, in view of the pairing between ${\mathcal C}_1(\H)$ and ${\mathcal K}(\H)$ (Eq. (\ref{eq:pairing})),
\[\phi(A) = \bar\lambda U A^\dag A^\ast\mbox{ for every }A\in {\mathcal K}(\H).\]

\bibliographystyle{elsarticle-num}

\begin{thebibliography}{00}

  \bibitem{maria}
  Maria Inez Cardoso Gon\c calves, Ahmed Ramzi Sourour, {\em
  Isometries of a generalized numerical radius,}
  Linear Algebra Appl. \textbf{429} (2008), no. 7, 1478--1488. 

\bibitem{FJ}
R.J. Fleming and J.E. Jamison, {\em Isometries on Banach spaces: function spaces,}
Chapman \& Hall/CRC Monographs and Surveys in Pure and Applied Mathematics, \textbf{129}. Chapman \& Hall/CRC, Boca Raton, FL, 2003. x+197 pp.

\bibitem{hausdorff}
F. Hausdorff, {\em Der Wertvorrat einer Bilinearform,} Math. Z. \textbf{3} (1919), 314--316.

\bibitem{kubrusly}
Carlos S. Kubrusly, {\em Spectral theory of operators on Hilbert spaces,} BirkhŠuser/Springer, New York, 2012. 

  \bibitem{kuzmaLiRodman}
  Bojan Kuzma, Chi-Kwong Li, Leiba Rodman , {\em Tracial numerical ranges and linear dependence of operators,} Electron. J. Linear Algebra, 22 (2011).
  
  \bibitem{lesnjak} G. Le\v snjak, {\em Additive preservers of numerical range,} Linear Algebra Appl. \textbf{345} (2002), 235--253.
  
  \bibitem{LS} C.K. Li, P. \v Semrl, {\em Numerical radius isometries,} Linear and Multilinear Algebra \textbf{50(4)} (2002), 307--314.
  
\bibitem{LMR}
C.-K. Li, P.P. Mehta, L. Rodman, {\em A generalized numerical range: the range of a constrained sesquilinear form,}
Linear and Multilinear Algebra \textbf{37} (1994), 25--49. 

\bibitem{milman}
D. Mil'man, {\em
Characteristics of extremal points of regularly convex sets,} 
Doklady Akad. Nauk SSSR (N.S.) \textbf{57} (1947), 119--122 (in Russian). 

\bibitem{OS}
M. Omladi\v c, P. \v Semrl, {\em Additive mappings preserving operators of rank one,} Linear Algebra Appl. \textbf{182} (1993),
239--256.

\bibitem{sakai} S. Sakai, {\it $C^\ast$-Algebras and $W^\ast$-Algebras,}
Springer-Verlag, Berlin--Neidelberg--NY, 1971; Reprinted,
Springer, 1998.

\bibitem{toeplitz}
O. Toeplitz, {\em Das algebraische Analogon zu einem Satz von Fej\'er,} Math. Z. \textbf{2} (1918), 187--197.

 \end{thebibliography}

\end{document}